\documentclass[a4paper, 12pt]{article}
\usepackage[left=30truemm,right=30truemm,top=30truemm,bottom=30truemm]{geometry}
\usepackage{amsmath, amssymb}
\usepackage{mathrsfs} 
\usepackage{mathtools} 
\usepackage{bm}
\usepackage{xcolor}
\usepackage[all]{xy} 
\xyoption{poly}

\usepackage{hyperref}
\definecolor{yelloworange}{RGB}{255,148,0}
\hypersetup{
	colorlinks=true, 
	citecolor=blue, 
	linkcolor=brown, 
	urlcolor=yelloworange, 
	breaklinks=true, 
}

\makeatletter
\@addtoreset{figure}{section}
\@addtoreset{table}{section}
\makeatother

\usepackage{amsthm}
\usepackage[capitalize, nameinlink]{cleveref}
\theoremstyle{plain}
\newtheorem{theorem}{Theorem}[section]
 
\newtheorem{lemma}[theorem]{Lemma} 
\newtheorem{proposition}[theorem]{Proposition} 
\newtheorem{thmintro}{Theorem}
\crefname{thmintro}{Theorem}{Theorems}

\theoremstyle{definition}

\theoremstyle{remark}
\newtheorem{remark}[theorem]{Remark}

\usepackage{etoolbox} 
\newcommand{\bAlphabet}{
  \def\do##1{\expandafter\newcommand\csname b##1\endcsname{\mathbb{##1}}}
  \docsvlist{C,F,P,Q,R,Z}
}
\bAlphabet
\newcommand{\cAlphabet}{
  \def\do##1{\expandafter\newcommand\csname c##1\endcsname{\mathcal{##1}}}
  \docsvlist{A,D,O,T}
}
\cAlphabet

\newcommand{\id}{\mathrm{id}}
\newcommand{\End}{\operatorname{End}}
\newcommand{\Gal}{\operatorname{Gal}}
\newcommand{\N}{\operatorname{N}}
\newcommand{\Res}{\operatorname{Res}}
\newcommand{\rk}{\operatorname{rk}}
\newcommand{\Tr}{\operatorname{Tr}}

\title{Dynamical degrees of automorphisms \\of K3 surfaces with Picard number $2$}
\date{September 18, 2025}
\author{Yuta Takada}
\author{Yuta Takada\thanks{Mathematical Sciences, the University of Tokyo, 
Meguro Komaba 3-8-1, Tokyo, Japan; 
JSPS Research Fellow. Saarland University, Saarbr\"ucken 66123, Germany. 
}}
\begin{document}

\maketitle
\renewcommand{\thefootnote}{\fnsymbol{footnote}} 
\renewcommand{\thefootnote}{\arabic{footnote}} 

\begin{abstract}
We show that there exists an automorphism of a projective K3 surface with Picard number $2$ 
such that the trace of its action on the Picard lattice is $3$. 
Together with a result of K.~Hashimoto, J.~Keum and K.~Lee, we determine the set of 
dynamical degrees of automorphisms of projective K3 surfaces with Picard number $2$. 
\end{abstract}

\section{Introduction}
Let $S$ be a compact K\"ahler surface. 
The intersection form makes $H^2(S,\bZ)_f$ a unimodular lattice, where  
$H^2(S,\bZ)_f \coloneqq H^2(S,\bZ)/\textrm{torsion}$.
Let $\phi$ be an automorphism of $S$. 
Then, the induced homomorphism $\phi^*:H^2(S, \bZ)_f \to H^2(S, \bZ)_f$
is an isometry of the lattice $H^2(S, \bZ)_f$. 
Let $\lambda(\phi)$ denote the spectral radius of 
$\phi^*:H^2(S, \bC) \to H^2(S, \bC)$:
\[ \lambda(\phi) \coloneqq \max\{ |\mu| \mid \text{$\mu\in \bC$ is an eigenvalue of
$\phi^*:H^2(S,\bC) \to H^2(S,\bC)$} \}. 
\] 
By the Gromov--Yomdin theorem, 
the topological entropy of $\phi$ is given by $\log\lambda(\phi)$. 
We call the value $\lambda(\phi)$ the (first) \textit{dynamical degree} of $\phi$
(a general definition of dynamical degrees can be found in \cite{Dinh-Sibony2005}). 
According to S.~Cantat \cite{Cantat1999}, if a compact complex surface admits an 
automorphism $\phi$ with $\lambda(\phi) > 1$, then the surface is bimeromorphic to 
a torus, a K3 surface, an Enriques surface, or a rational surface.  

A \textit{K3 surface} is a compact complex surface $S$ such that it has a nowhere 
vanishing holomorphic $2$-form $\omega_S$ and its irregularity 
($= \dim H^1(S, \cO_S)$) is zero. Every K3 surface is K\"{a}hler \cite{Siu1983}. 
Let $S$ be a K3 surface. 
The submodule $P_S \coloneqq H^2(S, \bZ)\cap H^{1,1}(S)$ of $H^2(S, \bZ)$ is 
called the \textit{Picard lattice} or the \textit{N\'eron-Severi lattice}. 
The \textit{Picard number} $\rho_S$ is the rank of $P_S$. 
The second cohomology group $H^2(S,\bZ)$ is a free $\bZ$-module of rank $22$ and, 
equipped with the intersection form, forms an even unimodular lattice of signature $(3,19)$. 
We refer to an even unimodular lattice of signature $(3,19)$ as a \textit{K3 lattice}. 
It is known that such a lattice is unique up to isomorphism. 

Let $\phi$ be an automorphism of $S$. 
It is known that $\lambda(\phi)$ is equal to either $1$ or a \textit{Salem number}, 
that is, a real algebraic integer $\beta > 1$
such that $\beta^{-1}$ is a conjugate of $\beta$ and all of its conjugates other than 
$\beta$ and $\beta^{-1}$ have absolute value $1$. 
The statement also holds for an automorphism of a compact K\"ahler surface, 
see \cite[Theorem 3.2]{McMullen2002} and \cite[Theorem 2.5]{Oguiso2010}. 
A natural question is to determine which Salem numbers can appear as dynamical degrees 
of automorphisms of K3 surfaces. In this article, we determine the set  
\[ 
\cD \coloneqq \left\{ \lambda(\phi) \;\middle|\; 
\begin{tabular}{l}
\text{$\phi$ is an automorphism of a projective K3 surface} \\
\text{with Picard number $2$}
\end{tabular}
\right\}.
\]

Assume that $S$ is a projective K3 surface. 
Then $\lambda(\phi)$ is given by the spectral radius of the restriction $\phi^*|_{P_S}$. 
Hence, the degree (as an algebraic number) of $\lambda(\phi)$ is at most $\rho_S$. 
Suppose further that $\rho_S = 2$ and $\lambda(\phi) > 1$.   
Then $\lambda(\phi)$ must be a Salem number of degree $2$. 
Its minimal polynomial is $X^2 - (\lambda(\phi) + \lambda(\phi)^{-1})X + 1$, and 
in particular, $\lambda(\phi) + \lambda(\phi)^{-1} \in \bZ_{\geq 3}$.  
So, we consider the set 
\begin{equation}\label{eq:defofcT}
\cT \coloneqq \left\{ \lambda(\phi) + \lambda(\phi)^{-1} \;\middle|\; 
\begin{tabular}{l}
\text{$\phi$ is an automorphism of a projective K3 surface} \\
\text{with Picard number $2$}
\end{tabular}
\right\}
\end{equation}
rather than $\cD$. Note that 
$\cD = \{ (\tau + \sqrt{\tau^2 - 4})/2 \mid \tau\in \cT \}$. 

In their paper \cite{Hashimoto-Keum-Lee2020}, K.~Hashimoto, J.~Keum and K.~Lee study the cases
$\phi^*\omega_S = \omega_S$ and $\phi^*\omega_S = -\omega_S$, 
and prove the following theorem. 

\begin{theorem}[{\cite[Main Theorem]{Hashimoto-Keum-Lee2020}}]\label{th:HKL}
Let $\tau \geq 3$ and $\epsilon\in \{1, -1\}$. Put  
$A_1 = \bZ_{\geq 4}$ and $A_{-1} = \bZ_{\geq 4}\setminus \{5, 7, 13, 17\}$. 
The following are equivalent:
\begin{enumerate}
\item There exists an automorphism $\phi$ of a projective K3 surface $S$
with Picard number $2$ such that 
$\lambda(\phi) + \lambda(\phi)^{-1} = \tau$ and $\phi^*\omega_S = \epsilon\omega_S$. 
\item There exists $\alpha \in A_\epsilon$ such that $\tau = \alpha^2 - 2\epsilon$. 
\end{enumerate}
\end{theorem}

It follows from this theorem that 
\begin{equation*}
\{2\}\cup\{\alpha^2 - 2 \mid \alpha\in \bZ_{\geq 4}\}\cup
\{ \alpha^2 + 2 \mid \alpha\in \bZ_{\geq 4}\setminus \{5, 7, 13, 17\} \} \subset \cT,  
\end{equation*}
where $2\in \cT$ since $\lambda(\id) + \lambda(\id)^{-1} = 2$. 
This article fills the gap left by this inclusion. 

\begin{thmintro}\label{th:A}
There exists a projective K3 surface $S$ with Picard number $2$ and an automorphism 
$\phi$ of $S$ such that the characteristic polynomial of the induced homomorphism 
$\phi^*:H^2(S, \bC)\to H^2(S, \bC)$ is $(X^2 - 3X + 1)\cdot\Phi_{50}(X)$, 
where $\Phi_{50}(X)$ is the $50$th cyclotomic polynomial. 
\end{thmintro}
To prove this theorem, we will construct an isometry of a K3 lattice 
with characteristic polynomial $(X^2 - 3X + 1)\cdot\Phi_{50}(X)$ by 
using methods in \cite{McMullen2011, McMullen2016}. Then, it leads to the existence of 
the desired automorphism of a K3 surface by the Torelli theorem and 
surjectivity of the period mapping. 

Let $\phi$ be the automorphism in \cref{th:A}. 
Then $\lambda(\phi) + \lambda(\phi)^{-1} = 3$, which implies that 
$3 = 1^2 + 2 \in \cT$. Furthermore, we get $7 = 3^2 - 2 \in \cT$ because 
\[ \lambda(\phi^2) + \lambda(\phi^2)^{-1} 
= \lambda(\phi)^2 + \lambda(\phi)^{-2} 
= (\lambda(\phi) + \lambda(\phi)^{-1})^2 - 2 = 7. 
\]
Hence,  
\begin{equation}\label{eq:lowerboundofcT}
\{2\}\cup \{\alpha^2 - 2 \mid \alpha\in \bZ_{\geq 3}\}\cup
\{ \alpha^2 + 2 \mid \alpha\in \bZ_{\geq 1}\setminus \{2, 3, 5, 7, 13, 17\} \} \subset \cT.  
\end{equation}
Moreover, some further discussions combined with \cref{th:HKL} show that 
the inclusion is, in fact, an equality. 

\begin{thmintro}\label{th:B}
We have 
\[ \cT = \{2\}\cup \{\alpha^2 - 2 \mid \alpha\in \bZ_{\geq 3}\}\cup
\{ \alpha^2 + 2 \mid \alpha\in \bZ_{\geq 1}\setminus \{2, 3, 5, 7, 13, 17\} \}.  
\]
\end{thmintro}

\begin{remark}
P.~Reschke \cite{Reschke2012, Reschke2017} shows that every Salem number of degree $2$ is 
realized as the dynamical degree of an automorphism of a complex $2$-torus 
(in this case, the torus must be projective). 
Let $\psi$ be an automorphism of a $2$-torus $A$ whose dynamical degree is a given 
Salem number $\beta$ of degree $2$. Then $\psi$ gives rise to an automorphism $\phi$ 
of the Kummer surface obtained from $A$, and $\phi$ also has dynamical degree $\beta$, 
see \cite[\S4]{McMullen2002}. Hence, every Salem number of degree $2$ is realized as the 
dynamical degree of an automorphism of a projective K3 surface.  
However, we remark that the Picard number of a Kummer surface is greater than or 
equal to $16$.
\end{remark}

\begin{remark}
Let $S$ be a non-projective K3 surface and $\phi$ an automorphism of $S$ with 
$\lambda(\phi) > 1$. Then $\lambda(\phi)$ is a Salem number of degree at least $4$. 
Furthermore, if $d$ denotes the degree of $\lambda(\phi)$, then $\rho_S = 22 - d$.  

It is shown that every Salem number of degree $d$ is realizable as the dynamical degree 
of an automorphism of a non-projective K3 surface for $d = 4, 6, 8, 12, 14, 16$ by 
E.~Bayer-Fluckiger \cite{Bayer-Fluckiger2025}, and for $d = 20$ by the author \cite{Takada2023} 
(a simple characterization in the case $d=22$ is given in \cite{Bayer-Fluckiger--Taelman2020}). 
In particular, as a non-projective analog of \cref{th:B}, we have 
\[ \begin{split}
&\left\{ \lambda(\phi) \;\middle|\; 
\begin{tabular}{l}
\text{$\phi$ is an automorphism of a non-projective K3 surface} \\
\text{with Picard number $2$}
\end{tabular}
\right\} \\
&\quad = \{1\}\cup\{\text{the Salem numbers of degree $20$}\}.
\end{split} \]
\end{remark}

\section{Lattices}\label{sec:Lattices}
This section summarizes some known results on lattices. 
For more details on \textit{gluing} and \textit{twist}, we refer to \cite{McMullen2011}.  

\paragraph{Lattices and isometries}
A \textit{lattice} is a finitely generated free $\bZ$-module $L$ equipped with 
an inner product, i.e., a nondegenerate symmetric bilinear form $b:L \times L \to \bZ$. 
Let $L = (L, b)$ be a lattice. Its \textit{dual} is defined to be
\[ L^\vee \coloneqq \{ y \in L\otimes\bQ \mid b(y, x) \in \bZ  
\text{ for all $x\in L$ }\},  \]
where $b$ is extended on $L\otimes \bQ$ linearly. 
Since $b$ is $\bZ$-valued, we have $L \subset L^\vee$. 
The lattice $L$ is \textit{unimodular} if $L = L^\vee$. 
We say that $L$ is \textit{even} if $b(x,x)\in 2\bZ$ for all $x\in L$. 
The \textit{signature} of $L$ is the signature of the inner product space over $\bR$ 
obtained by extending scalars to $\bR$. 

Let $e_1, \ldots, e_d$ be a basis of $L$.
The $d\times d$ matrix $G \coloneqq (b(e_i, e_j))_{ij}\in M_d(\bZ)$ is called the 
\textit{Gram matrix} of $(L, b)$ with respect to $e_1, \ldots, e_d$. 
Its determinant $\det G$ does not depend on the choice of the basis. 
This integer is called the \textit{determinant} of $L$ and denoted by 
$\det L$ or $\det b$. 
It is clear that $L$ is unimodular if and only if $|\det G| = 1$, 
and that $L$ is even if and only if all diagonal entries of $G$ are even. 

Let $(L_1, b_1)$ and $(L_2, b_2)$ be two lattices. 
An \textit{isometry} $t:L_1 \to L_2$ is a homomorphism of $\bZ$-modules satisfying 
$b_2(t(x), t(y)) = b_1(x, y)$ for all $x, y \in L_1$.  
When $L_1 = L_2$, we say that $t$ is an isometry of $L_1$.  

\begin{theorem}\label{th:SquareCondition}
Let $L$ be an even unimodular lattice of rank $2n$, and let 
$F(X)\in \bZ[X]$ be the characteristic polynomial of an isometry of $L$. 
Then $|F(1)|$, $|F(-1)|$, and $(-1)^n F(1)F(-1)$ are squares. 
\end{theorem}
\begin{proof}
See \cite[Theorem 6.1]{Gross-McMullen2002} or \cite[Theorem 2.1]{Bayer-Fluckiger2025}. 
\end{proof}

\paragraph{Glue groups}
Let $L = (L, b)$ be a lattice. The quotient $L^\vee/L$ is called 
the \textit{glue group} or the \textit{discriminant group} of $L$, and 
denoted by $G(L)$. It is a finite abelian group of order $|\det L|$. In fact, if 
$d_1, \ldots, d_r\in \bZ_{\geq1}$ are the elementary divisors of its Gram matrix 
then $G(L)\cong (\bZ/d_1\bZ) \oplus \cdots \oplus (\bZ/d_r\bZ)$. 
For a prime $p$, let $G(L)_p$ denote the Sylow $p$-subgroup of $G(L)$. 
Then $G(L)$ decomposes as $G(L) = \bigoplus_p G(L)_p$ where $p$ ranges over the 
primes dividing $\det L$. 
We remark that if $p.G(L)_p = 0$ then $G(L)_p$ is isomorphic to the direct sum of 
finitely many copies of $\bZ/p\bZ$ and can be regarded as an $\bF_p$-vector space. 

For $x\in L^\vee$, we write $\bar{x} =  x + L \in G(L)$. The glue group admits 
a nondegenerate \textit{torsion form} $\bar{b}:G(L)\times G(L) \to \bQ/\bZ$ defined by 
\[ \bar{b}(\bar{x}, \bar{y}) = b(x, y) \mod 1. \]
Furthermore, when $L$ is even, the self-product $\bar{b}(\bar{x}, \bar{x})$ is 
well-defined as an element of $\bQ/2\bZ$. 
Any isometry $t$ of $L$ induces an automorphism $\bar{t}:G(L)\to G(L)$ 
that preserves the torsion form $\bar{b}$.

\paragraph{Gluing}
Let $L = (L, b)$ be a lattice, and $S$ a submodule of $L$. 
We say that $S$ is \textit{primitive} in $L$ if $L/S$ is torsion-free, 
or equivalently, if $S = L\cap \bQ S$ in $L\otimes\bQ$. We define 
$S^\perp \coloneqq \{ y \in L \mid b(y,x) = 0 \text{ for all $x\in S$} \}$. 
Note that such a submodule is always primitive. 

\begin{proposition}\label{prop:gluing}
Let $(L_1, b_1)$ and $(L_2, b_2)$ be two even lattices. 
Suppose that there exists an isomorphism $\gamma:G(L_1)\to G(L_2)$ such that  
\[ \overline{b_1}(\bar{x}, \bar{x}) 
+ \overline{b_2}(\gamma(\bar{x}), \gamma(\bar{x})) = 0 
\quad\text{in $\bQ/2\bZ$ for all $\bar{x}\in G(L_1)$}.  \]
Then there exists an even unimodular lattice $L$ such that 
$L_1$ and $L_2$ are primitive sublattices in $L$ with $L_2 = L_1^\perp$. 

Assume further that isometries $t_1$ and $t_2$ of $L_1$ and $L_2$ with 
$\gamma\circ \overline{t_1} = \overline{t_2} \circ \gamma$ are given. Then 
there exists an isometry $t$ of $L$ such that $t|_{L_1} = t_1$ and $t|_{L_2} = t_2$. 
\end{proposition}
\begin{proof}
See \cite[page 5]{McMullen2011}. 
\end{proof}

In \cref{prop:gluing}, we remark that if $(r_1, s_1)$ and $(r_2, s_2)$ are the 
signatures of $L_1$ and $L_2$ respectively then that of $L$ is given by 
$(r_1 + r_2, s_1 + s_2)$; and if $F_1$ and $F_2$ are the characteristic 
polynomials of $t_1$ and $t_2$ respectively then that of $t$ is given by $F_1F_2$.

\paragraph{Twists}
Let $L = (L, b)$ be a lattice, and $t:L\to L$ an isometry. 
We write $\bZ[t+t^{-1}]$ for the subring of $\End(L)$ generated by 
$t + t^{-1}$. Let $a\in \bZ[t+t^{-1}]$ be an element with $\det(a) \neq 0$. Then 
\[ b_a(x, y) = b(ax, y) \quad (x, y \in L) \]
defines a new inner product on $L$. The lattice $(L, b_a)$ is 
called the \textit{twist} of $L$ by $a$, and denoted by $L(a)$. 
We have $\det(L(a)) = \det(a)\det(L)$. 
The isometry $t$ of $L$ is also an isometry of $L(a)$. 
If $L$ is even, then so is $L(a)$, see \cite[Proposition 4.1]{McMullen2011}. 

\begin{theorem}\label{th:twist}
Let $(L,b)$ be a lattice, $t$ an isometry of $L$, and 
$F\in\bZ[X]$ the characteristic polynomial of $t$.  
Let $p$ be a prime not dividing $\det L$, and let $a\in \bZ[t + t^{-1}]$ be 
an element with $\det(a)\neq 0$ and $pL \subset aL$. 
Suppose that $F\bmod p \in \bF_p[X]$ is separable. Then 
\[ G(L(a)) \cong G(L)\oplus G(L(a))_p \quad \text{as $\bZ[t]$-modules} \]
and $p.G(L(a))_p = 0$. Moreover, regarding $G(L(a))_p$ as an $\bF_p$-vector space, 
the characteristic polynomial of $\bar{t}|_{G(L(a))_p}:G(L(a))_p \to G(L(a))_p$ 
is $\gcd(A(X) \bmod p, F(X) \bmod p)$, where 
$A(X)\in \bZ[X]$ is a polynomial satisfying $a = A(t)$. 
\end{theorem}
\begin{proof}
See \cite[Theorem 4.2]{McMullen2011}. 
\end{proof}

\section{Proof of the main theorems}
We prove \cref{th:A,th:B}. 
We refer to \cite{Kondo_K3surfaces} for fundamental results on K3 surfaces. 

\subsection{Proof of \cref{th:A}}
Put $Q(X) = X^2 - 3X + 1$. 
To prove \cref{th:A}, we construct an isometry of a K3 lattice
with characteristic polynomial $Q\cdot \Phi_{50}$ by gluing. 
Let us begin by constructing two lattice isometries whose characteristic polynomials 
are $Q$ and $\Phi_{50}$ respectively, ensuring that they are compatible for gluing.  
In the following, the prime number $3001$ appears. It is a factor of the resultant 
$\Res(Q, \Phi_{50})$, and is referred to as a \textit{feasible prime} for $Q$ 
in \cite{McMullen2016}.

\begin{lemma}\label{lem:rank2lattice}
Let $L = (\bZ^2, b)$ be the lattice of rank $2$ having Gram matrix 
$3001 \begin{pmatrix}
2 & 1 \\
1 & -2
\end{pmatrix}$ with respect to the standard basis. 
Then, it has signature $(1, 1)$, and $G(L) \cong (\bZ/5\bZ)\oplus (\bZ/3001\bZ)^2$. 
Moreover $G(L)_5$ has a generator $\bar{v}$ such that 
$\bar{b}(\bar{v}, \bar{v}) = 2/5$ in $\bQ/2\bZ$. 
\end{lemma}
\begin{proof}
The former assertion is straightforward. 
Let $e_1, e_2$ denote the standard basis of $L$, and put 
$v = \frac{2}{5}e_1 + \frac{1}{5}e_2 \in L \otimes \bQ$.   
Then $b(v, e_1) = 3001$ and $b(v, e_2) = 0$. 
These imply that $v\in L^\vee$. Furthermore, $\bar{v}$ is a generator of 
$G(L)_5\cong \bZ/5\bZ$ since it has order $5$. Because $b(v, v) = 3001\cdot 2/5$, we have 
$\bar{b}(\bar{v}, \bar{v}) = 2/5$ in $\bQ/2\bZ$. 
\end{proof}

\begin{proposition}\label{prop:Q-lattice}
Let $L = (\bZ^2, b)$ be the lattice in Lemma \textup{\ref{lem:rank2lattice}}. 
Then $t = \begin{pmatrix}
1 & 1 \\
1 & 2
\end{pmatrix}$
is an isometry of $L$ with characteristic polynomial $Q(X)$. 
Moreover
\begin{enumerate}
\item $\bar{t}|_{G(L)_5} = - \id_{G(L)_5}$; and 
\item when we consider $G(L)_{3001}$ as an $\bF_{3001}$-vector space, 
the characteristic polynomial of $\bar{t}|_{G(L)_{3001}}$ is 
$Q(X) = (X+121)(X-124)\in \bF_{3001}[X]$. 
\end{enumerate}
\end{proposition}
\begin{proof}
We have 
\[ \begin{pmatrix}
1 & 1 \\
1 & 2
\end{pmatrix}^\top
\begin{pmatrix}
3001\cdot 2 & 3001\cdot 1 \\
3001\cdot 1 & 3001\cdot(-2)
\end{pmatrix}
\begin{pmatrix}
1 & 1 \\
1 & 2
\end{pmatrix}
= \begin{pmatrix}
3001\cdot 2 & 3001\cdot 1 \\
3001\cdot 1 & 3001\cdot(-2)
\end{pmatrix}, 
\]
which shows that $t$ is an isometry of $L$. 
It is clear that the characteristic polynomial of $t$ is $Q(X) = X^2 -3X + 1$. 

(i). Let $v =\frac{2}{5}e_1 + \frac{1}{5}e_2$ as in \cref{lem:rank2lattice}. Then 
\[ tv
= \begin{pmatrix}
1 & 1 \\
1 & 2
\end{pmatrix}
\begin{pmatrix}
2/5 \\ 
1/5 
\end{pmatrix}
= \begin{pmatrix}
3/5 \\ 
4/5 
\end{pmatrix}
= \begin{pmatrix}
1 \\ 
1 
\end{pmatrix}
- v. 
\]
This implies that $\bar{t}|_{G(L)_5} = - \id_{G(L)_5}$ since 
$\bar{v}$ generates $G(L)_5$. 

(ii). Let $L'$ be a lattice with Gram matrix 
$\begin{pmatrix}
2 & 1 \\
1 & -2
\end{pmatrix}$. 
Note that $t$ is also an isometry of $L'$ with characteristic polynomial $Q(X)$, 
and that $L = L'(3001)$. 
By applying \cref{th:twist} to $L'$ and $t$ with $p = a = 3001$, we see that 
the characteristic polynomial of $\bar{t}|_{G(L)_{3001}}$ is 
$Q(X) = (X+121)(X-124)\in \bF_{3001}[X]$. 
\end{proof}

Let $\Psi_{50}(Y)\in \bZ[Y]$ be the trace polynomial of $\Phi_{50}(X)$, 
that is, a unique polynomial of degree $10 = \deg(\Phi_{50})/2$ such that 
$\Phi_{50}(X) = X^{10}\Psi_{50}(X + X^{-1})$. 
Put $\zeta \coloneqq \exp\left(\frac{2\pi\sqrt{-1}}{50}\right)$ and 
$K \coloneqq \bQ(\zeta) \cong \bQ[X]/(\Phi_{50})$. 
The field $K$ admits an involution $\iota:K\to K$ uniquely determined by 
$\iota(\zeta) = \zeta^{-1}$. 
We define an inner product $b$ on the $\bQ$-vector space $K$ by 
\[b(x, y) = \Tr_{K/\bQ}\left(\frac{x\,\iota(y)}{\Psi_{50}'(\zeta + \zeta^{-1})}\right) 
\quad (x, y \in K), \]
where $\Psi_{50}'(Y) = \frac{d}{dY}\Psi_{50}(Y)$. 
Then, it follows from \cite[Theorem 8.1]{McMullen2002} that $L \coloneqq (\cO_K, b)$ is 
an even lattice with $|\det L| = |\Phi_{50}(1)\Phi_{50}(-1)| = 5$, 
where $\cO_K$ is the ring of integers of $K$, which is equal to $\bZ[\zeta]$. 
Furthermore, the linear transformation $t$ defined as the multiplication by $\zeta$
is an isometry of $L$ with characteristic polynomial $\Phi_{50}(X)$. 
Note that $\Phi_{50}(X) \bmod 3001$ is separable. In fact, 
for any positive integer $n$ and prime number $p$, $X^n - 1 \bmod p$ is separable if $p\nmid n$. 

We will use a twist of this lattice $L$. 
For some computations below, we make use of computer algebra. We define  
$k \coloneqq \{ x\in K \mid \iota(x) = x \} \subset K$. 
Then $k = \bQ(\zeta + \zeta^{-1})$, and $\bZ[t + t^{-1}] = \bZ[\zeta + \zeta^{-1}] = \cO_k$. 
Put 
\[ u_1 \coloneqq \zeta^2 + 1 + \zeta^{-2}, \quad 
u_2 \coloneqq \sum_{i=0}^5(\zeta^{2i+1} + \zeta^{-(2i+1)}), \quad 
a' \coloneqq \frac{\zeta + \zeta^{-1} - 3}{\zeta + \zeta^{-1} + 2},
\]
and 
\[ a \coloneqq u_1 u_2 a' =
 (\zeta^2 + 1 + \zeta^{-2})\cdot\left(\sum_{i=0}^5(\zeta^{2i+1} + \zeta^{-(2i+1)})\right)
\cdot\frac{\zeta + \zeta^{-1} - 3}{\zeta + \zeta^{-1} + 2}.  
\]
One can check that $u_1, u_2 \in \cO_k^\times$, and $a'\in \cO_k$. 
Hence $a\in \cO_k$. Moreover, we have
\begin{equation}\label{eq:N(a)}
\N_{k/\bQ}(a)
= \N_{k/\bQ}(a')
= \Psi_{50}(3)/\Psi_{50}(-2)
= 3001,  
\end{equation}
where $\N_{k/\bQ}$ is the norm map. 
This implies that $3001 \cO_k \subset a\cO_k$, and thus $3001 L \subset aL$. 
Therefore, we can apply \cref{th:twist} to the current $L$, $t$, and $a$, 
taking $p = 3001$.

\begin{proposition}\label{prop:Phi-lattice}
The twist $L(a) = (\cO_K, b_a)$ of $L$ by $a$ is an even lattice of signature $(2, 18)$. 
Moreover, the following assertions hold: 
\begin{enumerate}
\item $G(L(a)) \cong (\bZ/5\bZ)\oplus(\bZ/3001\bZ)^2$. 
\item $G(L(a))_5$ has a generator $\bar{v}$ such that 
$\overline{b_a}(\bar{v}, \bar{v}) = -2/5$ in $\bQ/2\bZ$. 
\item $\bar{t}|_{G(L(a))_5} = -\id_{G(L(a))_5}$. 
\item When we consider $G(L(a))_{3001}$ as an $\bF_{3001}$-vector space, 
the characteristic polynomial of $\bar{t}|_{G(L (a))_{3001}}$ is 
$Q(X) = (X+121)(X-124)\in \bF_{3001}[X]$. 
\end{enumerate}  
\end{proposition}
\begin{proof}
As mentioned in \S\ref{sec:Lattices}, the twist $L(a)$ is even since so is $L$. 
Furthermore, it follows from \cite[Theorem 4.2]{Gross-McMullen2002} that the positive signature 
$r$ of $L(a)$ is given by  
\[ r = 2\cdot\#\{ \sigma\in \Gal(k/\bQ) \mid
\sigma(a/\Psi_{50}'(\zeta + \zeta^{-1})) > 0 \}. \]
Using computer algebra, one can check that 
$\sigma\in \Gal(k/\bQ)$ with $\sigma(\zeta + \zeta^{-1}) = \zeta^7 + \zeta^{-7}$ 
is the only element satisfying $\sigma(a/\Psi_{50}'(\zeta + \zeta^{-1})) > 0$, 
see \cref{table:approx}. 
Hence $r = 2$, and the signature of $L(a)$ is $(2,18)$. 

\begin{table}[b]\centering
\begin{tabular}{c| r r r r r} \hline 
$\sigma(\zeta + \zeta^{-1})$ & $\zeta + \zeta^{-1}$ & $\zeta^3 + \zeta^{-3}$ & $\zeta^7 + \zeta^{-7}$ & $\zeta^9 + \zeta^{-9}$ & $\zeta^{11} + \zeta^{-11}$ \\
$\sigma(a/\Psi_{50}'(\zeta + \zeta^{-1}))$ & $-0.11372$ & $-0.067094$ & $0.028027$ & $-0.026605$ & $-0.11141$ \\
\hline
\end{tabular}

\vspace{0.5em}

\begin{tabular}{c| r r r r r}
\hline
$\sigma(\zeta + \zeta^{-1})$ & $\zeta^{13} + \zeta^{-13}$ & $\zeta^{17} + \zeta^{-17}$ & $\zeta^{19} + \zeta^{-19}$ & $\zeta^{21} + \zeta^{-21}$ & $\zeta^{23} + \zeta^{-23}$ \\
$\sigma(a/\Psi_{50}'(\zeta + \zeta^{-1}))$ & $-0.10565$ & $-0.029497$ & $-0.5185$ & $-1.5061$ & $-2.5493$ \\
\hline
\end{tabular}
\caption{Approximate values of $\sigma(a/\Psi_{50}'(\zeta + \zeta^{-1}))$. Note that 
$\sigma\in \Gal(k/\bQ)$ is determined by its value at $\zeta + \zeta^{-1}$.}
\label{table:approx}
\end{table}

(i). Since $|\det L| = 5$, we have $G(L) = \bZ/5\bZ$. Moreover, we have 
$\N_{K/\bQ}(a) = \N_{k/\bQ}(a^2) = 3001^2$ by \eqref{eq:N(a)}, and thus, 
\[ |\det(L(a))| = |\det(a)\det(L)| = |\N_{K/\bQ}(a)\det(L)| = 3001^2 \cdot 5. \]
Hence, the assertion follows from the former statement of \cref{th:twist}. 

(ii). Put 
\begin{multline*}
v = \frac{1}{5}(2 + 3\zeta + 2\zeta^2 + 3\zeta^3 + 2\zeta^4 + \zeta^5 - \zeta^6 
+ \zeta^7 - \zeta^8 + \zeta^9 + \zeta^{10}\\ 
- \zeta^{11} + \zeta^{12} - \zeta^{13}
+ \zeta^{14} - 3\zeta^{15} - 2\zeta^{16} - 3 \zeta^{17} - 2\zeta^{18} - 3\zeta^{19}).  
\end{multline*}
Then, using computer algebra, one can verify that $b_a(v, \zeta^j)$ is an integer 
for $j = 0, \ldots, 19$, and $b_a(v,v) = -142/5$. Hence, $v\in L(a)^\vee$, and
$\bar{v}$ is the desired generator of $G(L(a))_5$. 

(iii). We have $(t + \id_{L(a)^\vee})^{20}.L(a)^\vee \subset 5L(a)^\vee$ since 
$\Phi_{50}(X) \bmod 5 = (X+1)^{20}$. Thus 
$(\bar{t}|_{G(L(a))_5} + \id_{G(L(a))_5})^{20}.G(L(a))_5 = 0$, which leads to 
\[ (\bar{t}|_{G(L(a))_5} + \id_{G(L(a))_5}).G(L(a))_5 = 0 \] 
since $G(L(a))_5$ is a one-dimensional $\bF_5$-vector space. 
Hence $\bar{t}|_{G(L(a))_5} = -\id_{G(L(a))_5}$. 

(iv). Note that $(X+121)(X-124)\mid(\Phi_{50}(X) \bmod 3001)$ in $\bF_{3001}[X]$. 
Let $A(X)\in\bZ[X]$ be a polynomial satisfying $a = A(\zeta)$. Since 
\[  a' 
= \frac{\zeta + \zeta^{-1} - 3}{\zeta + \zeta^{-1} + 2} 
= \frac{\zeta^2 - 3\zeta + 1}{\zeta^2 + 2 \zeta + 1}
= \frac{Q(\zeta)}{(\zeta + 1)^2},
\]
we can write
\[ A(X) = U_1(X)U_2(X)h(X)Q(X) \mod \Phi_{50}(X), \] 
where $U_i(X)$ is a polynomial with $u_i = U_i(\zeta)$ ($i = 1,2$) and  
$h(X)$ is the inverse of $(X+1)^2$ in $\bZ[X]/(\Phi_{50})$. 
Hence $(Q \bmod 3001)$ is a common divisor of $(A \bmod 3001)$ and $(\Phi_{50} \bmod 3001)$ 
in $\bF_{3001}[X]$. Since $G(L(a))_{3001}$ is two-dimensional over $\bF_{3001}$, 
The latter assertion of \cref{th:twist} shows that the characteristic polynomial of 
$\bar{t}|_{G(L(a))_{3001}}$ is $Q(X)$. 
\end{proof}

\begin{remark}
One can show \cref{prop:Phi-lattice} by using matrix representation. 
For reference, if $G = (g_{ij})_{ij}$ is the Gram matrix of $L(a)$ with respect to the 
basis $1, \zeta, \zeta^2, \ldots, \zeta^{19}$, then each entry $g_{ij}$ depends only on 
the difference $i-j$, and the first row of $G$ is given by
\[ (-10, 8, -6, 3, -1, -2, 3, -3, 3, -3, 3, -3, 3, -3, 3, -3, 3, -3, 3, -3). \]
\end{remark}

We glue the two lattices in \cref{prop:Q-lattice,prop:Phi-lattice}.

\begin{theorem}\label{th:QPhi-lattice}
There exists an isometry $t$ of a K3 lattice $\Lambda$ 
with characteristic polynomial $Q\Phi_{50}$ such that 
$\Lambda$ contains a $t$-invariant primitive sublattice $L_1$ having Gram matrix 
$3001
\begin{pmatrix}
2 & 1 \\
1 & -2
\end{pmatrix}$ 
and the characteristic polynomial of $t|_{L_1}$ is $Q$. 
\end{theorem}
\begin{proof}
Let $(L_1, b_1)$ and $t_1$ (resp. $(L_2, b_2)$ and $t_2$) be the lattice and its 
isometry in \cref{prop:Q-lattice} (resp. \cref{prop:Phi-lattice}). 
There exists an isomorphism $\gamma_5:G(L_1)_5 \to G(L_2)_5$ with 
$\overline{b_1}(\bar{x}, \bar{x}) + \overline{b_2}(\gamma_5(\bar{x}), \gamma_5(\bar{x})) = 0$ 
in $\bQ/2\bZ$ for any $\bar{x}\in G(L_1)_5$ by \cref{lem:rank2lattice} and 
(ii) of \cref{prop:Phi-lattice}. Furthermore 
$\gamma_5\circ \overline{t_1}|_{G(L_1)_5} 
= - \gamma_5 
= \overline{t_2}|_{G(L_2)_5} \circ \gamma_5$. 
On the other hand, it follows from \cite[Theorem 3.1]{McMullen2011}
together with (ii) of \cref{prop:Q-lattice} and (iv) of 
\cref{prop:Phi-lattice} that there exists an isomorphism 
$\gamma_{3001}:G(L_1)_{3001}\to G(L_2)_{3001}$ such that 
$\overline{b_1}(\bar{x}, \bar{x}) 
+ \overline{b_2}(\gamma_{3001}(\bar{x}), \gamma_{3001}(\bar{x}))
= 0$ in $\bQ/2\bZ$ for any $\bar{x}\in G(L_1)_{3001}$ and 
$\gamma_{3001}\circ \overline{t_1}|_{G(L_1)_{3001}} 
= \overline{t_2}|_{G(L_2)_{3001}} \circ \gamma_{3001}$. 
Hence, the resulting isometry $t$ obtained by applying 
\cref{prop:gluing} to $\gamma_5\oplus \gamma_{3001}:G(L_1) \to G(L_2)$
is the desired one. The proof is complete. 
\end{proof}

We are now ready to prove \cref{th:A}. 

\begin{proof}[Proof of Theorem \textup{\ref{th:A}}]
Let $t$ be the isometry of a K3 lattice $\Lambda$ in \cref{th:QPhi-lattice}. 
Put $L_2 = L_1^\perp$. Note that $L_2$ is also $t$-invariant, and 
the characteristic polynomial of $t|_{L_2}$ is $\Phi_{50}$. 
The signature of $L_2$ is $(2, 18)$ since that of $L_1$ is $(1,1)$. 
This implies that there exists a root $\delta$ of $\Phi_{50}$ such that 
the signature of the subspace 
$\{v\in L \otimes \bR \mid (t^2 - (\delta + \delta^{-1})t + 1).v= 0 \}$ 
is $(2, 0)$. Let $\omega \in L\otimes\bC$ be a nonzero eigenvector of $t$ corresponding 
to $\delta$. Then, by surjectivity of the period mapping 
(see \cite[Theorem 6.9]{Kondo_K3surfaces}), 
there exists a K3 surface $S$ and an isomorphism $\alpha:H^2(S, \bZ)\to \Lambda$ 
of lattices such that $\alpha(\omega_S) = \omega$. We have 
\[ \alpha(P_S)
= \alpha(\{x\in H^2(S, \bZ) \mid \langle x, \omega_S \rangle = 0 \})
= L_2^\perp 
= L_1,  \] 
which implies that $\rho_S = \rk L_1 = 2$. It also implies that $P_S$ is indefinite, 
and thus, $S$ is projective (see \cite[Proposition 4.11]{Kondo_K3surfaces}). 
Moreover, $\alpha^{-1}\circ t \circ \alpha$ preserves the K\"ahler cone 
because $P_S \cong L_1$ has no element with self-intersection number $-2$.  
Hence, by the Torelli theorem (see \cite[Theorem 6.1]{Kondo_K3surfaces}), there exists 
an automorphism $\phi$ of $S$ such that $\phi^* = \alpha^{-1}\circ t \circ \alpha$. 
This completes the proof. 
\end{proof}

\subsection{Proof of \cref{th:B}}
Let $S$ be a projective K3 surface, and $\phi:S\to S$ an 
automorphism of $S$. 
We write $F(X)$ and $g(X)$ for the characteristic polynomials of 
$\phi^*:H^2(S)\to H^2(S)$ and $\phi^*|_{P_S}:P_S\to P_S$. 
It follows from \cite[Corollary 8.13]{Kondo_K3surfaces} that $F$ can be written as 
\begin{equation}\label{eq:F=gPhi^m}
F(X) = g(X)\Phi_l(X)^m 
\end{equation}
for some $l$ and $m \in \bZ_{\geq 1}$, where
$\Phi_l(X)$ is the $l$-th cyclotomic polynomial. 

We now assume that $\rho_S = 2$. Then, Equation \eqref{eq:F=gPhi^m}
implies that $m\varphi(l)  = 22 - \deg g = 20$, 
where $\varphi(l): = \#(\bZ/l\bZ)^\times = \deg \Phi_l$. 
Hence, $m$ and $\varphi(l)$ are divisors of $20$, and in particular, 
\[ l \in \{1,2, 3, 4, 6, 5, 8, 10, 12, 11, 22, 25, 33, 44, 50, 66\}. \] 
We further assume that the dynamical degree $\lambda(\phi)$ is greater than $1$. 
Then $\lambda(\phi)$ is a Salem number of degree $2$, and we can write 
$g(X) = X^2 - \tau X + 1$, where 
$\tau \coloneqq \lambda(\phi) + \lambda(\phi)^{-1} \in \bZ_{\geq 3}$.  

\begin{proposition}\label{prop:sq_upm2}
We have $l = 1, 2, 5, 10, 25,$ or $50$. Moreover, putting
\begin{equation}\label{eq:epsilon}
\epsilon = \begin{cases}
  1 & \text{if $l = 1, 5, 25$} \\
  -1 & \text{if $l = 2, 10, 50$,}
\end{cases}
\end{equation}
the integer $\tau + 2\epsilon$ is a square number, and 
if $l \neq 1,2$ then $5(\tau - 2\epsilon)$ is a square number. 
\end{proposition}
\begin{proof}
Since $H^2(S,\bZ)$ is an even unimodular 
lattice, it follows from \cref{th:SquareCondition} that 
\[ |F(1)| = (\tau - 2)|\Phi_l(1)|^m \quad\text{and}\quad 
|F(-1)| = (\tau + 2)|\Phi_l(-1)|^m
\]
are squares. 
Suppose that $m=20$. Then $\varphi(l) = 1$, which means that $l = 1$ or $2$. 
If $l = 1$ then $|F(-1)| = (\tau + 2)\cdot |-2|^{20}$, and 
$\tau + 2$ must be a square. 
If $l = 2$ then $|F(1)| = (\tau - 2)\cdot 2^{20}$, and 
$\tau - 2$ must be a square. 
Suppose then that $m < 20$. 
If $m$ were even then both $\tau + 2$ and $\tau - 2$ would be squares, 
but it is impossible. Hence $m$ is odd, which yields $m = 1$ or $5$ 
since $m$ is a divisor of $20$. 

Suppose that $m = 5$. Then $\varphi(l) = 4$, which means that 
$l \in \{ 5, 8, 10, 12\}$. 
If $l$ were $8$ then 
$|F(1)| = (\tau - 2)\Phi_{8}(1)^5 = (\tau-2)\cdot 2^5$ and 
$|F(-1)| = (\tau + 2)\Phi_{8}(-1)^5 = (\tau + 2) \cdot 2^5$, and 
therefore $2(\tau-2)$ and $2(\tau+2)$ would be squares. 
However, such a $\tau\in \bZ_{\geq 3}$ does not exist. 
Hence $l\neq 8$. Similarly, we can get $l \neq 12$. 
If $l = 5$ then 
$|F(1)| = (\tau - 2)\Phi_{5}(1)^5 = (\tau-2)\cdot 5^5$ and 
$|F(-1)| = (\tau + 2)\Phi_{5}(-1)^5 = \tau + 2$, and hence 
$\tau + 2$ and $5(\tau - 2)$ are squares. 
Similarly, in the case $l = 10$, we can show that 
$\tau - 2$ and $5(\tau + 2)$ are squares. 

Suppose that $m = 1$. Then $\varphi(l) = 20$, which means that 
$l \in \{ 25, 33, 44, 50, 66 \}$. By similar arguments as above, 
the cases $l = 33$, $44$, and $66$ are ruled out, and it can be shown that: 
if $l = 25$ then $\tau + 2$ and $5(\tau - 2)$ are squares; and 
if $l = 50$ then $\tau - 2$ and $5(\tau + 2)$ are squares.  
The proof is complete. 
\end{proof}

Let $\epsilon\in \{1, -1\}$ be as in \cref{prop:sq_upm2}, and let 
$\alpha\in \bZ_{\geq 0}$ be the non-negative integer satisfying
\begin{equation*}
\tau + 2\epsilon = \alpha^2. 
\end{equation*}
Since $\tau \geq 3$, if $\epsilon = 1$ then $\alpha \geq 3$, and if $\epsilon = -1$ then
$\alpha \geq 1$. This implies that 
\begin{equation}\label{eq:upperboundofcT}
\cT \subset \{2\}\cup \{\gamma^2 - 2 \mid \gamma\in \bZ_{\geq 3}\}\cup
\{ \gamma^2 + 2 \mid \gamma\in \bZ_{\geq 1} \},  
\end{equation}
where $\cT$ is defined in \eqref{eq:defofcT}. Note that the sets 
$\{\gamma^2 - 2 \mid \gamma\in \bZ_{\geq 3}\}$ and 
$\{ \gamma^2 + 2 \mid \gamma\in \bZ_{\geq 1} \}$ are disjoint.

\begin{lemma}\label{lem:rulingout}
Suppose that $\epsilon = -1$. Then $\alpha \notin \{2, 3, 5, 7, 13, 17\}$. 
\end{lemma}
\begin{proof}
We have $l = 2$, $10$, or $50$ since $\epsilon = -1$.  
If $l = 2$ then $\phi^*\omega_S = -\omega_S$, and thus
$\alpha \notin \{2, 3, 5, 7, 13, 17\}$ by \cref{th:HKL}. 
Suppose that $l = 10$ or $50$. Then, $5(\tau + 2) = 5(\alpha^2 + 4)$ 
is a square number by \cref{prop:sq_upm2}. 
In particular, $\alpha^2 + 4$ is divisible by $5$. 
Hence, $\alpha$ cannot be any of $2, 3, 5, 7, 13, 17$. 
\end{proof}

\begin{proof}[Proof of Theorem \textup{\ref{th:B}}]
By \eqref{eq:upperboundofcT} and \cref{lem:rulingout}, we have 
\[
\cT \subset \{2\}\cup \{\alpha^2 - 2 \mid \alpha\in \bZ_{\geq 3}\}\cup
\{ \alpha^2 + 2 \mid \alpha\in \bZ_{\geq 1}\setminus\{ 2, 3, 5, 7, 13, 17\} \}. 
\] 
Together with the reverse inclusion \eqref{eq:lowerboundofcT}, 
this completes the proof of \cref{th:B}. 
\end{proof}

\paragraph{Acknowledgments}
This work is supported by JSPS KAKENHI Grant Number JP24KJ0044.

\bibliographystyle{plain}
\bibliography{refs} 
\end{document}